\documentclass[11pt]{amsart}

\RequirePackage{amsthm,amssymb,amstext,amscd,amsfonts,amsbsy,amsxtra,latexsym} 
\RequirePackage{hyperref}

\newcommand{\field}[1]{\mathbb{#1}} 
 
\newcommand{\Z}{\field{Z}} 
\newcommand{\R}{\field{R}} 
\newcommand{\C}{\field{C}} 
\newcommand{\bbC}{\field{C}} 
\newcommand{\bbR}{\field{R}} 
 
\newcommand{\Log}{\mathrm{Log}}
\newcommand{\sgn}{\mathrm{sgn}}
\def\bbQ{\field{Q}}
 
\def\bbH{\mathbb{H}}
\def\bbZ{\mathbb{Z}}

\def\I{{i}}
\renewcommand{\(}{\left(}
\renewcommand{\)}{\right)}

\newtheorem{theorem}{Theorem}[section]
\newtheorem{lemma}[theorem]{Lemma}
\newtheorem{proposition}[theorem]{\textbf{Proposition}} 
\newtheorem{corollary}[theorem]{Corollary}
\theoremstyle{remark}
\newtheorem*{remark}{Remark}

\theoremstyle{definition}
\newtheorem{definition}{\textbf{Definition}} 

\def\sddprod{{\vrule  height6pt depth0pt width0.4pt \! \times}}

\title{Massive Deformations of Maass forms and Jacobi forms}

\author{Marcus Berg}
\author{Kathrin Bringmann}
\author{Terry  Gannon}

\address{Marcus Berg, Department of Physics,  Karlstad University,  651 88 Karlstad, Sweden}
\email{marcus.berg@kau.se}
\address{Kathrin Bringmann, Department of Mathematics and Computer Science, University of Cologne, Weyertal 86-90, 50931 Cologne, Germany}
\email{kbringma@math.uni-koeln.de}
\address{Terry Gannon, Department of Mathematics, University of Alberta, Edmonton, Canada T6G 2G1}
\email{tjgannon@ualberta.ca}

\thanks{The research of the second author is supported by the Alfried Krupp Prize for Young University Teachers of the Krupp foundation. This project has received funding from the European Research Council (ERC) under the European Union's Horizon 2020 research and innovation programme (grant agreement No. 101001179). The research of the third author is supported in part by an NSERC Discovery Grant.}

\begin{document}
\begin{abstract}
We define  one-parameter ``massive'' deformations of Maass forms
and Jacobi forms. This is inspired
by  descriptions of plane gravitational waves in string theory. Examples include
massive Green's functions (that we write in terms of Kronecker--Eisenstein series) and massive modular graph functions. 
\end{abstract}
\maketitle

\section{Introduction and statement of results}
This note introduces a class of deformations of some classical objects (Maass forms and real-analytic
Jacobi forms), that naturally arise in string theory and in condensed-matter theory. It has been known for a long time that string amplitudes from surfaces (worldsheets) in flat spacetime transform ``nicely'' with respect to the mapping class group of the surface. For example, torus chiral blocks produce so-called vector-valued modular forms for the full modular group SL$_2(\bbZ)$. The plane (gravitational) wave  is a natural one-parameter deformation of flat spacetime. Any 
spacetime reduces to this type of spacetime in a specific limit \cite{Penrose}. This should lead to one-parameter deformations of e.g.\  (vector-valued)
modular forms, which is of independent interest in mathematics.
Similarly, in \cite{SI}, it was discussed how the statistical mechanics of theories at a critical point
can allow interesting one-parameter deformations away from criticality. 
All the above work in physics raises the question on how special these deformations are and whether one can develop a mathematical theory of these deformations that is somehow parallel
to that of modular or automorphic forms. 

The precursor to our classes of functions are some open string amplitudes which Bergman, Gaberdiel, and Green \cite{BGG} computed in the plane wave background.  
For example, for $t\in\mathbb R^+$ 
and $m \in \R$,  define
\begin{align}\label{Fm}
\mathcal F_m(t) &:=e^{-2 \pi c_m t}\left(1-e^{-2 \pi m t}\right)^{\frac12}\prod_{n\ge1} \left(1-e^{-2 \pi t \sqrt{m^2+n^2}}\right),
\end{align}
where 
\begin{align*}
c_{m} &:=\frac1{(2\pi)^2} \sum_{\ell\geq 1}\int_0^\infty e^{-\ell^2x-\frac{\pi^2m^2}{x}}dx\; . 
\end{align*}
The authors of \cite{BGG} established in Appendix A the following inversion formula:
\begin{equation}\label{Translaw}
\mathcal F_m (t)=\mathcal F_{mt}\left(\tfrac{1}{t}\right)\,.
\end{equation}
 Moreover, 
 they showed that $\lim_{m \rightarrow 0 } c_m = \tfrac{1}{24}$ and hence 
we may view  $\frac{\mathcal F_m(t)}{\sqrt{2\pi t m}}$ as 
a one-parameter deformation of the \textit{Dedekind eta function} given by $\eta(\tau) :=q^{\frac{1}{24}}\prod_{n\geq1} (1-q^n)$ ($q:=e^{2\pi i\tau}$), i.e.,
we obtain $\lim_{m\to 0}  \frac{\mathcal F_m(t)}{\sqrt{2\pi mt}}=\eta(it)$. Here, $m$ represents spacetime curvature,
thus $m\rightarrow 0$ is the flat-space limit and this is called a \emph{massive deformation}. 

It is remarkable that such a complicated-looking deformation of the Dedekind eta function preserves the inversion property of $ \eta. $ According to string theory, this occurs because the space of complex structures (up to equivalence) on an annulus is naturally parametrized by $t\in\bbR^+$ modulo the inversion $t\mapsto \frac 1t$. This suggests that to get deformations of modular forms which retain modularity with respect to the full modular group SL$_2(\bbZ)$, one should look at closed string (torus) amplitudes in the plane wave background. Some of these were computed by e.g.\ Takayanagi
\cite{Ta}; we 
 list those amplitudes, and then repackage them in  Section 2 into our prototypical example $\mathcal E_{1,\mu}(z;\tau)$. Our paper originated with the challenge of finding a natural mathematical interpretation  for  these torus amplitudes. In Section \ref{massivefunctions}, by studying the special properties of $\mathcal E_{1,\mu}(z;\tau)$, we define the new classes of functions. 

\noindent $\bullet$ A
{\it massive Maass form} is a smooth function $f_\mu(\tau)$ on $(\tau,\mu)\in
\bbH\times \bbR^+$ which, for each fixed $\mu$, transforms like a  modular form,
has at most polynomial growth
towards the cusps, and is annihilated by some differential operator, given in \eqref{1stDOf}. 

\noindent $\bullet$ A {\it massive Maass--Jacobi form} is a smooth  function $\phi_{\mu}(z;\tau)$ on $(z,\tau,\mu)\in\bbC\times \bbH
\times\bbR^+$ which,  for each fixed $\mu$ and $z$, transforms like a Jacobi form, has at most polynomial growth
towards the cusps, and which is annihilated by 
two  differential operators, defined in \eqref{2ndDOF} and \eqref{3rdDOF}.

In Section 3.2 we construct families of examples, and  and study them in Section 4.
Recently there has been extensive work in the physics literature on related
topics, some of which we review in Section \ref{compare}.

\section*{Acknowledgements}
The authors thank the referees for their useful comments.

\section{String theory torus amplitudes revisited}

\subsection{The plane wave torus amplitudes}
According to string theory,  the functions $\mathcal F_m$ given in \eqref{Fm} live on the moduli space of the cylinder, which is why it respects modular inversion but not translation. To find something
closer to a modular form, we  instead consider closed string amplitudes on a torus.

One of the most basic objects in physics is the one-loop partition function. This can 
be computed by functional determinants, as reviewed for example in Section 13.4.1 of \cite{book}. 
We begin with the partition function of a fermion (with twisted boundary conditions)
in a plane gravitational wave, as in equation (2.37) of  \cite{Sugawara:2002rs}
\begin{multline}  \label{Zmuz}
Z_{\alpha,\beta,m}\left(\tau\right) \\
:= e^{-8 \pi  c_{\alpha,m}\tau_2}\prod_{n\in\bbZ}\prod_{\pm}\left(1- e^{-2\pi \tau_2 \sqrt{m^2 + \left( n \pm \alpha \right)^2}+ 2 \pi i \left(n \pm \alpha \right)\tau_1\pm 2 \pi  i \beta}\right), 
\end{multline}
with  $\alpha$,$\beta \in \R$ and 
\begin{equation*}
c_{\alpha,m} :=\frac1{(2\pi)^2} \sum_{\ell \geq 1} \cos( 2\pi \ell \alpha)\!\! \int_0^\infty e^{-\ell^2x-\frac{\pi^2m^2}{x}}dx=\frac{m}{2\pi} 
\sum_{\ell\ge 1}\cos(2\pi \ell \alpha) \frac{K_1\left(2\pi  \ell m\right)}{\ell}.
\end{equation*}
Here and throughout the paper we write $\tau:=\tau_1+i\tau_2$ with $\tau_1\in\R$ and $\tau_2\in\R^+$. The \textit{$K$-Bessel function} enters through (see 10.32.10 of \cite{Ni}), $\text{Arg}(w)< \tfrac \pi 4$
\begin{equation}\label{BesselK}
K_\nu(w):=\tfrac 12 \left(\tfrac w2 \right)^\nu \int_0^\infty \exp\left(-x- \tfrac {w^2}{4x} \right) x^{-\nu -1} \text{d}x.
\end{equation}
In the literature, $Z_{\alpha,\beta,m}(\tau)$ is often written as $Z_{\alpha,\beta,m}(\tau,\overline{\tau})$ to emphasize its non-holomorphicity.
Note that $Z_{\alpha,\beta,m}(\tau)$ only depends on $\alpha,\beta$ modulo one.

In Appendix A of  \cite{Ta}, $Z_{\alpha,\beta,m}(\tau)$ was shown to obey
\begin{equation} \label{modinvZ}Z_{\alpha,\beta,m}(\tau+1)=Z_{\alpha,\alpha+\beta,m}(\tau)\,,\ \ Z_{\alpha,\beta,m}\left(-\tfrac{1}{\tau}\right)=Z_{\beta,-\alpha,\frac{m}{|\tau|}}(\tau)\,.\end{equation}
This modular covariance is quite remarkable, given the unfamiliar square root in the exponent of \eqref{Zmuz}, but it is required by string theory.
Computing that $\lim_{m\rightarrow 0} c_{\alpha , m} =\frac{\alpha^2}{4}- \frac{\alpha}{4}+\frac{1}{24}$, we find that 
\begin{align} \label{Zlimit1} 
\lim_{m\rightarrow 0}Z_{\alpha, \beta, m}(\tau) = e^{-2\pi  \alpha^2 \tau_2 } \left|  \frac{\vartheta_1(\alpha\tau+\beta;\tau) }{ \eta(\tau)} \right|^2,
\end{align}
where for $z \in \C$ the \textit{Jacobi theta function} is defined as
\begin{equation*}
\vartheta_1(z;\tau):=-2q^{\frac18} \sin (\pi z) \prod_{n=1}^\infty(1-q^n)(1-e^{2\pi i z}q^{n})(1-e^{-2\pi i z}q^n).
\end{equation*}
 In \cite{Sugawara:2002rs}, $ Z_{\alpha,\beta,m}(\tau) $ was called a massive theta function; we, however,  do
not use this name since it is not holomorphic even in the limit \eqref{Zlimit1}.

 \subsection{Reformulation of the torus amplitudes}
 \label{E1first}
In the inversion formula in \eqref{modinvZ}, 
it is inconvenient that the  parameter $m$ is rescaled; for this reason, the deformation parameter that we use in this paper is the SL$_2(\bbZ)$-invariant quantity $\mu:=m^2|\tau|$.
In physics terms,  if $m^2$ is the mass and $|\tau|$ is the area, then $\mu$ is dimensionless.
The invariant mass parameter $\mu$ was not used in
the early literature on strings in pp-waves. It was formally introduced in \cite{DHoker:2015gmr} when writing a generating function for modular graph functions, but there was no connection to the plane wave amplitudes.

 Just as Jacobi theta functions with characteristics can be recast as Jacobi forms, \eqref{Zlimit1} suggests that we can interpret $Z_{\alpha,\beta,m}(\tau)$ as a function $Z_\mu(z;\tau)$ in $z=\alpha\tau+\beta$ for $\mu$ fixed. Then the aforementioned periodicity in $\alpha,\beta$, together with \eqref{modinvZ}, implies that for each fixed $\mu$, $Z_\mu(z;\tau)$  transforms under $\mathrm{SL}_2(\bbZ)\sddprod\bbZ^2$ like a Jacobi form of index and weight zero.
  
This raises the question what other basic property $Z_\mu(z;\tau)$ possesses. The final ingredient in our repackaging is that the logarithm of a partition function of a free theory can be a {Green's function}. The \emph{Green's functions} $G(x,y)$ of a differential operator $L$ are the solutions to an inhomogeneous (partial) differential equation
$L(x)G(x,y)=\delta(x-y)$, where $\delta(x)$ is the Dirac delta distribution; more on this is explained below equation \eqref{LaplaceG}. At least formally,
a solution
 to the differential equation $L(f)=g$ is then $f(x)=\int G(x,y)g(y)dy$. 
In quantum field theory (or string theory) the Green's function for the 
Laplace equation (or Weyl equation, or whatever the equation of motion for the system is)  plays a central role e.g.
to calculate Feynman diagram integrals. Thus we may hope that $\operatorname{Log}(Z_\mu)$ satisfies a simple differential equation (of course it continues to transform with index and weight zero). Putting all of this together, define  $$\mathcal E_{1,\mu}(z;\tau):=-{\rm Log} \left( Z_{-\alpha, \beta,\sqrt{\frac{\mu}{|\tau|}}}(\tau)\right)\,.$$
Comparing \eqref{Zlimit1} with \eqref{Estheta} then gives that\footnote{Here we are starting from the
partition function of a fermion. If we would  instead start from the partition function of a boson in equation (2.37) of  \cite{Sugawara:2002rs}, then we would arrive  at a different sign.}
 \begin{equation} \label{Zlimit2} 
\lim_{\mu\rightarrow 0^+}\mathcal E_{1,\mu}(z;\tau) =\mathbb E_1(z;\tau) \,,
\end{equation}
with the \textit{Kronecker--Eisenstein series}
\begin{equation}
\mathbb E_1(z;\tau):=
{\tau_2 \over \pi}\sum_{(r, \ell)\in\Z^2\setminus\{(0,0)\}} {e^{{2\pi i\over \tau_2}{\rm Im}((r\tau+\ell)\overline{z})} \over |r\tau+\ell|^{2}}.
\label{EsEarly}
\end{equation}
Note that $\mathbb{E}_1(z;\tau)$ is up to a factor the Green's function of the Laplace
equation on the torus   (see e.g.  \cite[Chapter 7.2]{Polchinski:1998rq})

\begin{equation} \label{LaplaceG}
\partial_{z}\partial_{\overline{z}}( {\mathbb E}_1(z;\tau)) = -2\pi \delta^{[2]}(z;\tau) +{\pi \over \tau_2}\,,
\end{equation}
where, for a variable $w$, we set $\partial_w:=\frac{\partial}{\partial w}$. 
Here, the Dirac delta distribution $\delta^{[2]}(z;\tau)$ is associated to the linear functional on the  space of smooth functions on the torus $\bbC/(\bbZ\tau+\bbZ)$ sending such a function $f(z)$ to $f(0)\in\bbC$. We formally write this linear functional as the integral operator given by 
$\int\int_P f(z)\delta^{[2]}(z;\tau) \, d^2z $ where   $P$ is a fundamental domain on $\bbC$.
Since the left-hand side integrates to zero on a compact space, the constant term
 $\pi \over \tau_2$ on the right-hand side of \eqref{LaplaceG}  is
needed to cancel the integral  $\int\!\! \int_P \delta^{[2]}(z;\tau) \, d^2 z=1$. Incidentally, if we allow for quasiperiodicity $w$
as in $\mathbb E_1(w,z;\tau)$ in  \eqref{E2}, then the constant term  $\pi \over \tau_2$ in \eqref{LaplaceG} is absent, unless 
$w\in\bbQ\tau+\bbQ$.

To summarize, in this subsection we repackage the string  torus amplitudes as
  $\mathcal E_{1,\mu}(z;\tau)$, which can be thought of as a deformed non-holomorphic
 Jacobi form of weight and index zero, and more precisely as a massive deformation of the torus Green's function of the Laplace equation. We identify in the next section what differential equation $\mathcal E_{1,\mu}(z;\tau)$
satisfies.

Using \eqref{Estheta} we rewrite \eqref{Zlimit1} as
\begin{equation*} \label{E1mulimit}
\lim_{\mu \rightarrow 0^+} \mathcal{E} _{1, \mu} (z; \tau) =\mathbb E_1(z;\tau)=  - 2  \log  \left| \frac{\vartheta_1(z;\tau)}{ \eta(\tau)} \right| +  \frac{2 \pi z_2 ^2}{\tau_2}\, ,
\end{equation*}
where $ z=z_1+iz_2 $. Since $\vartheta_1(z;\tau)$ vanishes at the lattice points $z\in\Z\tau+\Z$,  we see that $\mathbb{E}_1$ has logarithmic singularities 
at lattice points. 
These logarithmic singularities are crucial
in mathematical physics, when $\mathbb E_1(z;\tau)$ is used as a Green's function: it
is  precisely what
is required to obtain $\delta^{[2]}$ in \eqref{LaplaceG} above; readers unfamiliar with this may want to consult the
textbook Problem 2.1 in \cite{Polchinski:1998rq} with the solution in \cite{Headrick:2008ke}.

\section{Massive Jacobi forms}

\subsection{The definitions}
\label{direct}
\label{massivefunctions}

 We note that holomorphy is much more restrictive than modularity. To be more precise, it is easy to construct functions which  are non-holomorphic but modular-invariant.
For example letting $g(z)$ be any smooth function on the sphere and  $J(\tau)$ 
the SL$_2(\bbZ)$-Hauptmodul, the composition $g(J(\tau))$ is smooth, modular-invariant, and bounded at the cusps. Of course, the way to proceed is to replace the conditions like $\partial_{\overline{\tau}}=0$ by higher order differential equations invariant under the appropriate Lie group.

We first consider functions $f_\mu(\tau)$ on  $\bbH\times \bbR^+$. The group SL$_2(\bbR)$ acts on $\bbH$ as usual and fixes $\bbR^+$. The \textit{weight $k$ hyperbolic Laplace operator} for SL$_2(\bbR)$ is defined as
\begin{equation}\Delta_{\tau,k} :=-\tau^2_2\left(\partial_{\tau_1}^2+\partial_{\tau_2}^2\right)+ik\tau_2(\partial_{\tau_1}+i\partial_{\tau_2})=-4\tau_2^2\partial_\tau\partial_{\overline{\tau}}+2ik\tau_2\partial_{\overline{\tau}}.
\end{equation}
We say that a function $g$ has polynomial growth towards $i\infty$ if
\begin{equation*}
g(\tau) = O \left( \tau_2^{a} \right) \quad \text{as } \tau_2 \to \infty
\end{equation*}
for some $a>0$. Decay towards the other cusps is defined similarly. 

\begin{definition}
\label{massivestuff}
A \textit{massive Maass form } of weight $k$ is a smooth function $f_\mu(\tau)$ on $(\tau,\mu)\in
\bbH\times \bbR^+$ that 
 transforms like a  modular form of weight $k$ for some Fuchsian group, satisfies
\begin{equation}\label{1stDOf}
\Delta_{\tau,k}(f_\mu)(\tau)=\left(g_2(\mu)\partial_\mu^2+g_1(\mu)\partial_\mu+g_0(\mu)\right) f_\mu(\tau)
\end{equation}
for certain smooth functions $g_j:\bbR^+\to\bbC$,  and  for each fixed $\mu\in\bbR^+$, has
polynomial growth towards the cusps. 
 If $f_\mu$ is a massive Maass form, and $f(\tau):=\lim_{\mu \to 0^+}f_\mu(\tau)$ exists for all $\tau\in\bbH$, then we say  $f_\mu$ is a \textit{massive deformation} of $f$. 
\end{definition}

Any (holomorphic) modular form $f(\tau)$ has massive deformations, e.g. $f(\tau)+a\mu$ for $a\in\bbC$. 
In Theorem \ref{Maassex} below, we give massive Maass forms $\mathcal E_{s,\mu}(0;\tau)$ of weight zero which, for any $ s\in\mathbb C $ with  ${\rm Re} (s)>1$, are massive deformations of the Kronecker--Eisenstein series $\mathbb E_s(0;\tau)$ defined in \eqref{Es}.

 Note that adding $\lambda f_\mu$ to the right-side of \eqref{1stDOf} is unnecessary as it could be absorbed into $g_0$. 
Also, observe that we can remove the first order term $g_1(\mu)\partial_{\mu}$ in \eqref{subsec:exp}
by making the change of variable $\mu=g(\nu)$, provided that we find a solution
to  the auxiliary ordinary differential equation $g_1(g(\nu)) g'(\nu) - g_2(g(\nu)) g''(\nu)=0$.

 If $f_\mu(\tau)$ is a massive deformation of a Maass form $f(\tau)$, then for any smooth function $g(\mu)$ on $\bbR^+$ with  $\lim_{\mu\to0^+}g(\mu)=1$, $g(\mu)f_\mu(\tau)$ is  another massive deformation of $f(\tau)$, though with different $g_j(\mu)$. 
 Thus if there is one massive deformation of a given Maass form, then there are many.

Next consider  functions $\phi_\mu(z;\tau)$ on  $\bbC\times\bbH\times \bbR^+$. The Jacobi group SL$_2(\bbR)\sddprod\bbR^2$ acts on $\bbC\times\bbH$ as usual and fixes $\bbR^+$.  In particular, restricted to $A=[\left({a\atop c}{b\atop d}\right), (\lambda,\mu)]\in\mathrm{SL}_2(\bbZ)\sddprod\bbZ^2$,  the \textit{weight $k$ index $m$ slash operator} is
\begin{multline*}
\phi|_{k,m}A(z;\tau):=(c\tau+d)^{-k}\exp\left(2\pi im\left(-c\frac{(z+\lambda\tau+\mu)^2}{c\tau+d}+\lambda^2\tau+2\lambda z\right)\right)\\
\times \phi\left(\frac{z+\lambda\tau+\mu}{c\tau+d};\frac{a\tau+b}{c\tau+d}\right).
\end{multline*}
 The Laplacian here is 
\begin{equation*}\Delta_{z,k,m}:=2\tau_2\partial_z\partial_{\overline{z}}+8\pi i \alpha\tau_2m\partial_{\overline{z}}-2\pi i m.\end{equation*}
The {\it Casimir operator} of order three is given by
\begin{multline*}
C_{k,m} := 
-4\tau_2z_2\left(\partial_{z}\partial^2_{\overline{z}}+\partial^2_{z} \partial_{\overline{z}}\right)
-4\tau_2^2\left(\partial_{\overline{\tau}}\partial^2_{z}+\partial_{\tau} \partial^2_{ \overline{z}}\right)+2ik\tau_2\left(\partial_{z}\partial_{\overline{z}} +\partial^2_{\overline{z}}\right)
\\
+4\pi i m \left(8\tau_2^2\partial_{\tau} \partial^2_{\overline{\tau}}-2z_2^2 \partial^2_{\overline{z}}+8\tau_2z_2\partial_{\tau} \partial_{ \overline{z}}-2i\left(2k-1\right)\tau_2\partial_{\overline{\tau}} +2kiz_2\partial_{\overline{z}}   \right).
\end{multline*}
This operator was introduced into the context of Maass--Jacobi forms in \cite{Pitale}.
There are several definitions of those forms in the literature, one example is given in \cite{KR}.

\begin{definition}
\label{massivestuff2}
A \textit{massive Maass--Jacobi form} is a smooth function $\phi_\mu(z;\tau)$ on $(z,\tau,\mu)\in
\bbC\times \bbH\times\bbR^+$   that 
 transforms like a Jacobi form of weight $k$ and index $m$ for some discrete subgroup of $ \mathrm{SL}_2(
\bbR)\sddprod\bbR^2$, satisfies both 
\begin{align}\label{2ndDOF}C_{k,m}\left(\phi_\mu\right)(z;\tau)=&\,\lambda \,\phi_\mu(z;\tau),\\ 
\label{3rdDOF} \Delta_{z,k,m}(\phi_\mu)(z;\tau)=&\,-\left(G_2(\mu)\partial_\mu^2+G_1(\mu)\partial_\mu+G_0(\mu)\right)\phi_\mu(z;\tau)
\end{align}
for certain smooth functions $G_0,G_1,$ and $G_2\in C^\infty(\bbR^+)\,$, and has polynomial growth
for each fixed $z$ and $\mu$. (The minus sign on the right is to facilitate comparison with Section  \ref{compare}.)
If $\phi_\mu$ is a massive Maass--Jacobi form, and $\phi(z;\tau):=\lim_{\mu \to 0^+}\phi_\mu(z;\tau)$ exists for all  $z\in\bbC$ and $\tau\in\bbH$, then we say that $\phi_\mu$ is a \textit{massive deformation} of $\phi$. 
\end{definition}

Note that a massive Maass--Jacobi form is, among other things, a one-parameter family of Maass--Jacobi forms. 
For example, we prove in Corollary \ref{corKE} that ${\mathcal E}_{1,\mu}(z;\tau)$ is a massive Maass--Jacobi form of weight and index zero, for SL$_2(\bbZ)\sddprod\bbZ^2$, for which $\lambda=G_1(\mu)=G_0(\mu)=0$ is a possible choice. We generalize this example significantly in Theorem \ref{Jacobiexamples}: e.g. in Corollary \ref{corKE} we give  a massive deformation $\mathcal E_{s,\mu}(z;\tau)$ of $\mathbb E_s(z;\tau)$ for all $ s\in\mathbb C $ with  ${\rm Re} (s)>1.$

\begin{lemma} \label{Lemmaequiv}\normalfont  Choose any smooth functions 
$ g,\varphi\in C^\infty(\bbR^+)$ such that both $g$ and $\varphi'$ never vanish and $\varphi$ maps $\bbR^+$ onto $\bbR^+$.

\begin{itemize}
	
\item[(a)] Suppose that $f_\mu(\tau)$ is a massive Maass form of weight $k$, for some discrete subgroup $G$ of $\mathrm{SL}_2(\bbR)$. Then $F_\mu(\tau):=g(\mu)f_{\varphi(\mu)}(\tau)$ is also a massive Maass form of weight $k$  for $G$.

\item[(b)] Suppose that $\phi_\mu(z;\tau)$ is a massive Maass--Jacobi form of weight $k$ and index $m$, for some discrete subgroup $\Gamma$ of  $\mathrm{SL}_2(\bbR)\sddprod\bbR^2$. Then $\Phi_\mu(z;\tau):=g(\mu)\phi_{\varphi(\mu)}(z;\tau)$ is also a massive Maass--Jacobi form of weight $k$ and index $m$ for $\Gamma$.\end{itemize}\end{lemma}

\begin{proof} (a) Clearly $F_\mu$ satisfies the same transformation law with respect to $G$ as $f_\mu$, and also has the same limiting behavior towards the cusps. To see that \eqref{1stDOf} holds, with different functions $g_0, g_1$, and $g_2$, we directly compute
\begin{eqnarray*}
\Delta_{\tau, k}\left(F_\mu\right)(\tau)
=\frac{g(\mu)}{\varphi^{'}(\mu)^2}\Big(g_2(\varphi(\mu))\partial_{\mu}^2
+\left(-\varphi^{''}(\mu)g_2(\varphi(\mu))+g_1(\varphi(\mu))\varphi^{'}(\mu)\right)\Big)\partial_{\mu}\\
+g_0(\varphi(\mu))\varphi^{'}(\mu)^2. 
\end{eqnarray*}
The proof of (b) is similar. 
\end{proof}

Lemma \ref{Lemmaequiv} makes the following definitions natural.

\begin{definition} \label{def:f} For functions $g(\mu),\varphi(\mu),f_\mu(\tau),F_\mu(\tau),\phi_\mu(z;\tau)$, and $\Phi_\mu(z;\tau)$ as in Lemma \ref{Lemmaequiv}, we call $f_\mu(\tau)$ and $F_\mu(\tau)$  \textit{equivalent} as massive Maass forms, and $\phi_\mu(z;\tau)$ and $\Phi_\mu(z;\tau)$  \textit{equivalent} as massive Maass--Jacobi forms.\end{definition}

It is easy to verify that the equivalences defined above indeed yield equivalence relations.

 Note that the above objects are truly doubly periodic for $\mu=0$. It is, however,  also of interest
to allow for quasiperiodicity, which appears as a parameter $w$ in Section \ref{alternative} below. 
This can be relevant for vector-valued modular forms.

\subsection{Examples}\label{subsec:exp}

We first note the well-known fact that the Kronecker--Eisenstein series ${\mathbb E}_{1}(z;\tau)$
is a Maass--Jacobi form of weight and index zero, as we  now review.
\begin{proposition}
\label{MJE1}
For $z\notin \Z\tau+\Z$, we have
\begin{equation*}
C_{0,0} \left({\mathbb E}_{1}(z;\tau)\right)=0.
\end{equation*} 
\end{proposition}
\begin{proof}
Write $C_{0,0}=C_{0,0;1}+C_{0,0;2}$ with  
\begin{equation}\label{defineC}
C_{0,0;1}:=-4\tau_2z_2\left(\partial_{z} \partial_{\overline{z}}^2+\partial^2_{z} \partial_{ \overline{z}}\right),\qquad C_{0,0;2}:=-4\tau_2^2\left(\partial_{\overline{\tau}}\partial^2_{z}+\partial_{\tau} \partial^2_{\overline{z}}\right).
\end{equation}
For ${\rm Re}(s)$ sufficiently large,  the double sum representation of $\mathbb E_s(z;\tau)$ 
in \eqref{Es} converges absolutely even after differentiation. We 
act on each summand, that we denote by  ${\mathbb E}_{s,r,\ell}(z;\tau)$, to obtain that
\begin{align*}
\tau_2 C_{0,0;1}({\mathbb E}_{s,r,\ell})(z;\tau)&=8\pi^3 iz_2|r \tau+\ell|^2r\, {\mathbb E}_{s,r,\ell}(z;\tau),\\
\tau_2 C_{0,0;2}({\mathbb E}_{s,r,\ell})(z;\tau)&=-8\pi^3 iz_2|r \tau+\ell|^2r\, {\mathbb E}_{s,r,\ell}(z;\tau)
\end{align*}
so the contributions from the two pieces of $C_{0,0}$ cancel termwise. The claim then follows via analytic continuation to $s=1$.
\end{proof}

To determine the action of the operator   $C_{0,0}$  on  the deformed $\mathcal E_{1,\mu}(z;\tau)$, it would
be convenient to have a similar double sum representation as for $\mathbb E_s(z;\tau)$ 
in \eqref{Es}.  
 For this purpose, recall the Bessel function defined in \eqref{BesselK}.
Note that $K_\nu(x)$ obeys the differential equation 
\begin{equation}
 x^2f''(x)+xf'(x)-\left(x^2+\nu^2\right)f(x).
 \label{BesselODE}
 \end{equation}
For $\nu\in\bbR^+$, $K_\nu$ has the asymptotic behavior $K_\nu(x)\sim \sqrt{\frac{\pi}{2x}}e^{-x}$ as $x\to\infty$, and $K_\nu(x)\sim \frac{1}{2}\Gamma(\nu)(\frac{x}{2})^{-\nu}$ as $x\to0^+$.

\begin{proposition}\label{3.3}
\label{E1muFourier}
We have
\begin{align*}
\mathcal E_{1,\mu}(z;\tau)&=
2\sqrt{\mu\tau_2} \sum_{(r,\ell)\in\mathbb Z^2\setminus\{(0,0)\}} \!\!{K_1\left(2\pi \sqrt{\frac{\mu}{\tau_2}}|r  \tau+ \ell|\right) 
 \over
|r  \tau+ \ell|} e^{ \frac{2 \pi i}{\tau_2} \text{\rm Im} (( r\tau + \ell) \overline z)}  
 \; . 
\end{align*}
In particular $\mathcal E_{1,0}(z;\tau)=\mathbb E_1(z;\tau)$. 
\end{proposition}

We defer the proof  of Proposition \ref{E1muFourier} to  Section \ref{fourierE1} and rather first consider a generalization.
Proposition \ref{E1muFourier} suggests the following generalization.
To state it, we let ${\mathcal S}$ denote the space of all functions $h\in C^\infty(\bbR^+)$, which are 
$O(e^{-ax})$ as $x\to\infty$, for some $a>0$, and $O(x^{-b})$ as $x\to 0$, for some $b>0$. 
For example, $K_\nu\in\mathcal S$ for each $\nu\in\bbR^+$. 
  Define
   \begin{multline}\label{E1reallygeneralized}
   	{\mathcal E}_{h,\mu,\left[a, b,c,d, L\right]}(z;\tau)
   	\\:=\frac{\mu^d}{\tau_2^c}\!\sum_{(r,\ell)\in\mathbb Z^2\setminus\{(0,0)\}}|r\tau+\ell|^{2c} \,h\left( \frac{\mu^a}{\tau_2^b}|r  \tau+ \ell|^{2b}\right) e^{ \frac{ 2 \pi i }{\tau_2} L\text{\rm Im} ((r \tau + \ell) \overline z)}.
\end{multline}
Then for any $h\in\mathcal S$, $a\in\bbR$, $b\in\bbR^+$, and $L\in\bbZ$,  the double-sum in ${\mathcal E}_{h,\mu, \left[a, b,c, d, L\right]}(z;\tau)$ converges absolutely to a smooth function in $\bbC\times\bbH\times\bbR^+ $. By Lemma \ref{Lemmaequiv}, ${\mathcal E}_{h,\mu, \left[a, b,c, d, L\right]}(z;\tau)$ is a massive Maass--Jacobi form if and only if ${\mathcal E}_{h,\mu, \left[1, b,c, 0, L\right]}(z;\tau)$ is.

\begin{theorem} \label{Jacobiexamples} Choose any $a\in\bbR^+$, $b\in\bbR^+$, $L\in\bbZ$, and any smooth function $h\in\mathcal S$ satisfying the differential equation
\begin{equation}
\label{otherODE}
x^2\,h''(x)+\gamma x\,h'(x)+\left(\kappa-\nu x^{\frac{1}{a}}\right)\,h(x)=0
\end{equation}
for some constants $\gamma,\kappa,\nu\in\bbR$, $\nu\ne 0$. Then
\begin{eqnarray*}\label{Ehab}{\mathcal E}_{h,\mu,[a,b, L]}(z;\tau)&:=&{\mathcal E}_{h,\mu,\left[1,a, b,0, L\right]}(z;\tau)\\
&=&
\!\!\!\!\!\!\!\!\sum_{(r,\ell)\in\mathbb Z^2\setminus\{(0,0)\}}  \!\!\frac{|r\tau+\ell|^{2b}}{\tau_2^b}\,h\left( \frac{\mu|r  \tau+ \ell|^{2a}}{\tau_2^a}\right) e^{ \frac{2 \pi i}{\tau_2} L\text{\rm Im} (( r \tau + \ell) \overline z)}
\end{eqnarray*}
is a massive Maass--Jacobi form  of weight and index zero. \end{theorem}

\begin{proof} First note that any ${\mathcal E}_{h,\mu,[a,b, c, d, L]}(z;\tau)$ transforms like a Jacobi form of weight and of index zero. To verify the cusp condition in Definition \ref{massivestuff2}, it suffices to consider $\tau\to i\infty$ thanks to invariance under SL$_2(\bbZ)$, and each term in \eqref{Ehab} exponentially decays to 0  because $h\in\mathcal S$,
except the $r=0$ terms, that have at most polynomial growth.

Because $h$ decays rapidly as $x\to\infty$, we can verify the partial differential equation \eqref{3rdDOF} term-by-term. Dropping the dependence on $a,b,L,$ and $h$ from the notation, we define 
\begin{equation*}f_{\mu,r,\ell}(\tau):=
\frac{|r\tau+\ell|^{2b}}{\tau_2^b}\,h\left( \frac{\mu|r  \tau+ \ell|^{2a}}{\tau_2^a}\right) 
\end{equation*}
and keep $E_{r,\ell,L}(z;\tau):=e^{ \frac{ 2 \pi i}{\tau_2} L \text{\rm Im} (( r\tau + \ell) \overline z)}$ as a separate factor. 
We obtain,  with $C_{0,0;1}$ and $C_{0,0;2}$ as defined in \eqref{defineC}
\begin{align*}
C_{0,0;1}(f_{\mu,r,\ell}E_{r,\ell,L})(z;\tau)=&\,8L^3\pi^3i\alpha r|r\tau+\ell|^2f_{\mu,r,\ell}(\tau)E_{r,\ell,L}(z;\tau)\,,\\
	C_{0,0;2}(f_{\mu,r,\ell}E_{r,\ell,L})(z;\tau)=&-4L^2\pi^2 \left((r\tau+\ell)^2\partial_{\tau}(f_{\mu,r,\ell})(\tau)\right.\\
		 & \left.+(r\overline{\tau}+\ell)^2\partial_{\overline{\tau}}\left(f_{\mu,r,\ell}(\tau)\right)\right)E_{r,\ell,L}(z;\tau)\\
		  & -8\pi^3r\alpha i L^3|r\tau+\ell|^2f_{\mu,r,\ell}(\tau)E_{r,\ell,L}(z;\tau)
		 		    \; .
\end{align*}
Since $((r\tau+\ell)^2\partial_{\tau}
		  +(r\overline{\tau}+\ell)^2\partial_{\overline{\tau}})F(\frac{|r\tau+\ell|^2}{\tau_2})$ vanishes for any smooth function $F$,  we have that $C_{0,0;2}(f_{\mu,r,\ell}E_{r,\ell,L})=-C_{0,0;1}(f_{\mu,r,\ell}E_{r,\ell,L})$. Hence \eqref{2ndDOF} holds.
		  
Turning to \eqref{3rdDOF},
 we obtain that
$$\Delta_{z,0,0}\left(f_{\mu,r,\ell}E_{r,\ell,L}\right)(z;\tau)=-2\pi^2L^2\mu^{-\tfrac{b+1}{a}}X_{\mu}(\tau)^{\tfrac{b+1}{a}}h(X_{\mu}(\tau))E_{r,\ell,L}(z;\tau),$$ where 
$X_\mu (\tau):=\frac{\mu|r\tau+\ell|^{2a}}{\tau_2^a}$.
 We also compute, for any functions $G_j(\mu)$ as above, 
\begin{multline*}
\left(G_2(\mu)\partial_\mu^2+G_1(\mu)\partial_\mu+G_0(\mu)\right)f_{\mu,r,\ell}(\tau)\\ =
G_2(\mu)\mu^{-2-\tfrac{b}{a}}X_{\mu}(\tau)^{\tfrac{b}{a}+2}h''(X_\mu(\tau))+G_1(\mu)\mu^{-1-\tfrac{b}{a}}X_\mu(\tau)^{\tfrac{b}{a}+1}h'(X_\mu(\tau))+\\
G_0(\mu)\mu^{-\tfrac{b}{a}}X_\mu(\tau)^{\tfrac{b}{a}}h(X_\mu(\tau)).
\end{multline*}
Choosing $G_0(\mu):=L^2\tfrac{2\pi^2\kappa}{\nu}\mu^{-\tfrac{1}{a}}$, $G_1(\mu):=L^2\tfrac{2\pi^2\gamma}{\nu}\mu^{1-\tfrac{1}{a}}$, $G_2(\mu):=L^2\tfrac{2\pi^2}{\nu}\mu^{2-\tfrac{1}{a}}$ and using \eqref{otherODE}, we obtain that \eqref{3rdDOF} holds termwise.  
\end{proof}
\begin{remark}
We see from the proof of Theorem \ref{Jacobiexamples} that the parameters $\gamma,\kappa,\nu\in\bbR$, $\nu\ne 0$
in \eqref{otherODE} can be used to  ``tune'' the ordinary differential equation \eqref{otherODE}. This means that if applications
demand a certain 
partial differential operator in $\mu$ in  \eqref{3rdDOF}, i.e., certain functions $G_0(\mu)$, $G_1(\mu)$, and $G_2(\mu)$, the ordinary differential equation \eqref{otherODE} adjusts accordingly. An explicit example is given in Section  \ref{modgraph}. 
By the change of variables $x\mapsto({w \over 2b\sqrt{\nu}})^{2b}$
and $h \mapsto x^{\frac{1-\gamma}{2}}H$, the differential equation \eqref{otherODE} for $h(x)$
transforms to the Bessel equation \eqref{BesselODE} for $H$ with $\nu^2=b^2((\gamma-1)^2-4\kappa).$ However, this
 change of variables may not be admissible in certain applications,
for example it  generically changes the massless limit,
so we treat  \eqref{otherODE} as a generalization of \eqref{BesselODE}. If one elects to extend Definition \ref{def:f} to allow further powers of $\partial_\mu$ on the right-side of  
\eqref{3rdDOF}, then Theorem \ref{Jacobiexamples} is extended to include  families 
for which
the corresponding \eqref{otherODE}  is  not a Bessel-type equation. For example, if $g_3(\mu)\partial_\mu^3$ is included in  \eqref{3rdDOF}, then the proof of Theorem \ref{Jacobiexamples} yields massive Maass--Jacobi forms for certain functions $h(x)$ satisfying $$x^3h'''(x)+\varrho x^2\,h''(x)+\gamma x\,h'(x)+\left(\kappa-\nu x^{\frac 1a}\right)\,h(x)=0$$ for $\varrho,\gamma,\kappa,\nu\in\bbR$.

\end{remark}

From Theorem \ref{Jacobiexamples}, it is natural to generalize Proposition \ref{E1muFourier} as in the following corollary. For this, we set
\begin{equation*}{\mathcal E}_{s,\mu}(z;\tau):= 
	2\sum_{(r,\ell)\in\mathbb Z^2\setminus\{(0,0)\}}  \!\!\left(\frac{\sqrt{\mu\tau_2}}{|r\tau+\ell|}\right)^s\,K_s\left(2\pi \sqrt{\frac{\mu}{\tau_2}}|r  \tau+ \ell|\right) e^{2\pi i(r\beta-\ell\alpha)}.
\end{equation*}
\begin{corollary} \label{corKE} For any $s\in\C$ with
		$\operatorname{Re}(s)> 0 $,  the function
${\mathcal E}_{s,\mu}(z;\tau) $
is a massive Maass--Jacobi form for $\mathrm{SL}_2(\bbZ)\sddprod\bbZ^2$, of weight and index zero. For $\mu\to 0^+$ we have
${\mathcal E}_{s,\mu}(z;\tau) \to \mathbb E_{s}(z;\tau)$.
\end{corollary}

\begin{proof} We find that ${\mathcal E}_{s,\mu}(z;\tau)=2\,{\mathcal E}_{2K_s(2\pi x),\mu, [\frac12,\frac12,\frac{-s}{2},\frac{s}{2},1]}(z;\tau)$, and therefore, by Theorem \ref{Jacobiexamples},  ${\mathcal E}_{s,\mu}(z;\tau)$ is a massive Maass--Jacobi form for $\mathrm{SL}_2(\bbZ)\sddprod\bbZ^2$, of weight and index zero. 
Taking the limit as $\mu\to 0^+$ and comparing with \eqref{Es} below, we obtain $\mathbb E_{s}(z;\tau)$ as desired.
\end{proof}

It is interesting that we get massive deformations for $\mathbb E_s$, even though $\Delta_{z,0,0}({\mathbb E}_s)$ is proportional to $\mathbb E_{s-1}$, and not to $\mathbb E_s$ (see \eqref{DeltaEs}). What makes this possible are the $G_j(\mu)$ in \eqref{3rdDOF}.

\begin{theorem}\label{Maassex} For any $h\in{\mathcal S}$, and any $a,c,d\in\bbR$ and $b\in\bbR^+$, with $a\ne 0$, the function 
${\mathcal E}_{h,\mu[{a, b,c, d, 0}]}(0;\tau)$
 is a massive Maass form of weight zero for {\rm SL}$_2(\bbZ)$. 
 \end{theorem}

    \begin{proof}  Theorem \ref{Jacobiexamples} yields that each ${\mathcal E}_{h,\mu,\left[{a, b,c, d, L}\right]}(0;\tau)$ transforms like a weight zero modular form for SL$_2(\bbZ)$ and satisfies the cusp condition, so all that remains is to verify \eqref{1stDOf}. By Lemma \ref{Lemmaequiv}, it suffices to take $a=1$ and $d=0$. 
    Writing $f_{\mu,r,\ell}(\tau):= \frac{|r\tau+\ell|^{2c}}{\tau_2^c}h(\frac{\mu|r\tau+\ell|^{2b}}{\tau_2^b})$, we compute 
    
\begin{multline*}
\Delta_{\tau,0}\left(f_{\mu,r,\ell}\right)(\tau)=
-b^2\mu^2\frac{|r\tau+\ell|^{4b+2c}}{\tau_2^{2b+c}}h''(X_{\mu}(\tau))\\
-\left(b^2+2bc+b\right)\mu\frac{|r\tau+\ell|^{2b+2c}}{\tau_2^{b+c}}h'(X_{\mu}(\tau))
-\left(c^2+c\right)\frac{|r\tau+\ell|^{2c}}{\tau_2^c}h(X_{\mu}(\tau))
\end{multline*}
with $X_{\mu}(\tau):=\frac{\mu}{\tau_2^b} |r\tau+\ell|^{2b}$.
Note that this can also be written as 
\begin{multline*}
\Delta_{\tau,0}\left(f_{\mu,r,\ell}\right)(\tau)=-\mu^{-{c \over b}}X_{\mu}(\tau)^{c \over b}
\left(b^2 X_{\mu}(\tau)^2 h''\left(X_{\mu}(\tau)\right)\right.\\
\left.+b(b+2c+1)X_{\mu}(\tau)h'\left(X_{\mu}(\tau)\right)+c(c+1)h\left(X_{\mu}(\tau)\right)\right).
\end{multline*}
Similarly to the proof of Theorem \ref{Jacobiexamples}, we use this
explicit expression for $\Delta_{\tau,0}\left(f_{\mu,r,\ell}\right)(\tau)$ to obtain that \eqref{1stDOf} is satsified with
$g_2(\mu):=-b^2\mu^2$, $g_1(\mu):=-\left(b^2+2bc+b\right)\mu,$ and $g_0(\mu):=-c^2-c$;
note that there is no condition on $h$ in this case. 
\end{proof}

 The  examples of massive Maass forms in Theorem \ref{Maassex} are often deformations of Maass forms. For example, if  $h(x)=2K_s(2\pi x)$ for $s\in\C$ with $\operatorname{Re}(s)>1$, then
  $\mathcal E_s(0;\tau)={\mathcal E}_{h,\mu,[\frac12,\frac12,\frac{-s}{2}, \frac{s}{2},0]}(0;\tau)$ is a massive deformation of $\mathbb{E}_s(0;\tau)$, as in Corollary \ref{corKE}.

\section{Fourier coefficients of $\mathcal E_{1,\mu}(z;\tau)$ and the proof of Proposition \ref{E1muFourier}}
\label{fourierE1}

\subsection{Proof of Proposition \ref{E1muFourier}}
We are now ready to compute the Fourier expansion of $\mathcal E_{1,\mu}(z;\tau)$. Note that in the massless case $\mu=0$, calculating the Fourier expansion of $\mathcal E_{1,\mu}(z;\tau)$ 
is equivalent  to proving Kronecker's second limit formula, as expressed in Appendix A; in physics, this is the Fourier expansion of the  Green's function of the Laplace equation on the torus, see Problem 7.3 in \cite{Polchinski:1998rq}.
\begin{proof}[\textnormal{\bf Proof of Proposition \ref{E1muFourier}}]
Taking the logarithm of equation \eqref{Zmuz} and recalling that
$ \mathcal E_{1,\mu}(z;\tau)$ is defined with $-\alpha$ gives
\begin{equation*}
\mathcal E_{1,\mu}(z;\tau) =8 \pi c_{\alpha,\mu}\tau_2 +
 \sum_{n\in\Z} \sum_{\pm} \Log \left(1- e^{-2\pi \tau_2 \sqrt{\frac{\mu}{\tau_2} + \left( n \pm \alpha\right)^2}+ 2 \pi i \left(n \pm \alpha\right)\tau_1\mp 2 \pi i \beta}\right) \; . 
\end{equation*}

Using Poisson summation, the $(\ell,r)$-th ($\ell,r\in\Z$) Fourier coefficient of the second term equals
\begin{align}\label{Poisson}
&\int_0^1 \!\! \int_0^1  \!
e^{-2\pi i (\ell \alpha-r\beta)} \sum_{n\in\Z} \sum_{\pm}
 \Log \left(1- e^{-2\pi \tau_2 \sqrt{\frac{\mu}{\tau_2}+ \left( n \pm \alpha\right)^2}+ 2 \pi i \left(n \pm \alpha \right)\tau_1\mp 2 \pi i \beta}\right)
 d\alpha d\beta \nonumber\\
&\quad=\sum_{\substack{n\in\Z \\ j\geq 1 \\ \pm}} \frac{1}{j} \int_0^1 e^{2\pi i(-\ell\alpha + (n\pm \alpha)j\tau_1) -2\pi j \tau_2\sqrt{\frac{\mu}{\tau_2}+(n\pm\alpha)^2}}d\alpha  \int_0^1 e^{2\pi i(-r\mp j)\beta} d\beta,
\end{align}
inserting the series expansion of the logarithm.
 The integral on $\beta$ now vanishes unless $j=\mp r$ in which case it equals $1$. Since $j\geq 1$ we have no solution if $r=0$. We thus assume for the remaining calculation that $r\neq 0$. We obtain that $\mp = \sgn(r)$ and $j=|r|$, thus \eqref{Poisson} equals
\begin{equation}\label{rewriteint}
\frac{1}{|r|} \sum_{n\in\Z} \int_0^1 e^{2\pi i(-\ell \alpha + (n-\sgn(r)\alpha)|r|\tau_1)-2\pi |r|\tau_2 \sqrt{\frac{\mu}{\tau_2}+(n-\sgn(r)\alpha)^2}}d\alpha.
\end{equation}
Noting that $e^{-2\pi i\ell \alpha}$ is invariant under $\alpha\mapsto \alpha +\sgn(r)n$,  \eqref{rewriteint} becomes 
\begin{equation}\label{aint}
\frac{1}{|r|} \int_\R e^{-2\pi i(\ell+r\tau_1)\alpha-2\pi |r| \tau_2\sqrt{\frac{\mu}{\tau_2}+\alpha^2}} d\alpha.
\end{equation}

We next use (26) on page 16  of \cite{Bateman}, which states that for $A,B \in \C$ with $\text{Re}(A), \text{Re}(B) >0$ 
\begin{equation*}
\int_0^\infty e^{-B \sqrt{x^2+ A^2}} \cos(xy) \text{d}x= \frac{AB}{\sqrt{y^2+ B^2}} K_1 \left(A \sqrt{y^2+ B^2} \right).
\end{equation*}
Thus \eqref{aint} becomes
\begin{equation*}
2 \sqrt{\mu \tau_2} \frac{1}{|r\tau+ \ell|} K_1 \left(2 \pi  \sqrt{\tfrac{\mu}{\tau_2}}  |r\tau + \ell| \right).
\end{equation*}
This yields
\begin{equation*}
\mathcal E_{1, \mu}(z;\tau) = 8 \pi c_{\alpha, \mu} \tau_2 + 2\sqrt{\mu_2 \tau_2}\sum_{\substack{r \in \Z \setminus \{0\} \\ \ell \in \Z}} \frac{K_1 \left( 2 \pi \sqrt{\tfrac{\mu}{\tau_2}}  |r\tau + \ell |\right)}{|r \tau + \ell|} e^{2 \pi i ( r \beta-\ell \alpha)}.
\end{equation*}
Plugging in the definition of $c_{\alpha,\mu}$ then gives the claimed Fourier expansion. 
The $c_{\alpha,\mu}$ term furnishes the $r=0$ terms that exhibit the expected polynomial growth
towards the cusp $i\infty$. 

The limit in Proposition \ref{E1muFourier} is clear using that $\lim_{x \to 0} x K_1(x)=1$.
\end{proof}
\subsection{The Mellin transform}
\label{mellintransform}
The Mellin transform of
Jacobi-form-like objects (or, vector-valued modular forms) is interesting to consider since it
can produce the corresponding Dirichlet series, or automorphic L-functions. 
Let us briefly review this, and then apply it to
our massive Jacobi forms. 
The {\it Mellin transform} (see 2.5.1 of \cite{Ni}) of a locally integrable function $f$ is a Laplace-like transform
\begin{equation*}
\mathcal M(f)(s) := \int_0^{\infty} f(x) x^{s-1} dx
\end{equation*}
that is analytic in some strip $a< {\rm Re} (s) < b$,
with its inverse (see  2.5.2 of \cite{Ni}) obtained by integration along a vertical line
shifted by any constant $c$ in the strip $a<c<b$,
\begin{equation*}
f(x)={1 \over 2\pi i}\int_{c-i\infty}^{c+i\infty}\mathcal M(f)(s) x^{-s} ds \; . 
\end{equation*}

Let us consider the Mellin transform of $\mathcal E_{1,\mu}(z;\tau)$ with respect to $\mu$
and call the Mellin-dual variable $s$. We have
\begin{proposition} \label{propmellin}
 For $s\in\C$ with Re$(s)>0$, the Mellin transform of $\mu \mapsto \mathcal E_{1, \mu}(z;\tau)$ equals 
\[
\mathcal{M}\left(\mathcal{E}_{1, .} (z; \tau)\right)(s)=
\frac{\Gamma(s)}{\pi^s} \mathbb E_{s+1}(0,z;\tau).
\]
\end{proposition}
\begin{proof}
Plugging in Proposition \ref{E1muFourier} gives that the Mellin transform of $\mu \mapsto \mathcal E_{1,\mu} (z;\tau)$ is 
\begin{equation*}
\int_0^\infty \!\!\!\!\mathcal E_{1,\mu} (z;\tau) \mu^{s-1} d\mu 
=\!\!\int_0^{\infty} \!\!  2\sqrt{\mu \tau_2}\!\!\!\!\!\!\!\! \!\!\!\!\sum_{(r,\ell) \in\Z^2\setminus\{(0,0)\}}\!\!\!\! \!\!\!\!\!\!\!\!{ K_1\left(2\pi \sqrt{\mu \over \tau_2}|r  \tau+ \ell|\right) \over
	|r  \tau+ \ell|}e^{2\pi i( r\beta-\ell \alpha)} \mu^{s-1} d\mu.
\end{equation*}
Since the sum
converges absolutely for $\mu>0$ (10.25.3 of \cite{Ni} gives
the exponential decay $K_1(x) \sim \sqrt{\pi \over 2x}e^{-x}$ for $ x \rightarrow \infty$), we may interchange summation and integration.
We now compute, making the change of variables $x=2\pi \sqrt{\tfrac {\mu}{\tau_2}} |r \tau+ \ell|$
\begin{align*}
\int_0^\infty K_1 \left( 2  \pi  \sqrt{\tfrac {\mu}{\tau_2}} |r \tau+ \ell| \right) \mu^{s- \tfrac 12} d \mu
= 2 \left( \frac{\tau_2}{4 \pi^2 |r \tau+\ell|^2} \right) ^{s+\tfrac 12} \int_0^\infty K_1(x) x^{2s} dx.
\end{align*}
Now 10.43.19 of \cite{Ni} states that for $a,b \in \C$ with $\text{Re}(a)<\text{Re}(b)$ we have
\begin{equation*}
\int_0^{\infty} K_a(x) x^{b-1} dx= 2^{b-2} \Gamma\left( \tfrac{b-a}{2} \right) \Gamma \left( \tfrac{a+b}{2} \right).
\end{equation*}
From this it not hard to conclude the claim.
\end{proof}
Note that as long as  $\text{Re}(a)<\text{Re}(b)$, it is straightforward to generalize Proposition
\ref{propmellin} to other $\mathcal E_{s,\mu}(z;\tau) $ than $\mathcal E_{1,\mu}(z;\tau) $. It is  surprising that an integral transform of the massive ($\mu>0$) Kronecker--Eisenstein series produces the  massless  
($\mu=0$) Kronecker--Eisenstein series.
This means that we  can take the inverse Mellin transform 
of the classical Kronecker--Eisenstein series to obtain our main example, namely
\begin{align} \label{invmellinZ}
\mathcal E_{1,\mu}(z;\tau) &=  {1 \over 2\pi i}\int_{c-i\infty}^{c+i\infty}  (\pi\mu)^{-s} \Gamma(s)\mathbb E_{s+1}(0;z; \tau)ds,
\end{align}
where as in Subsection \ref{mellintransform}, $c$ lies in the region of convergence, here the positive real axis. 
In Appendix A of \cite{Ta}, it was  proven 
that $\mathcal E_{1,\mu}(z;\tau)$ transforms like a Jacobi form of weight and index zero. Equation \eqref{invmellinZ} gives a reproof of this fact using the modular invariance of  the undeformed $\mathbb E_{s+1}(0;z;\tau)$ (see e.g.\ Section 4 of \cite{Siegel}), since $\mu$ is invariant.

Proposition \ref{propmellin} and its inverse equation \eqref{invmellinZ} may be useful in the following sense.
There is a vast literature in string theory
where integrals over sums of products
of Kronecker--Eisenstein series were performed. Particularly relevant examples
here include \cite{Broedel:2014vla,Kiritsis:1997em,Lerche:1987qk,Stieberger:2002wk}.
In the last few years, this long-standing theme in string theory  has been put in mathematical
terms as  modular graph functions \cite{DHoker:2015wxz},  we give an example
of this in Section \ref{modulargraph}.   Proposition \ref{propmellin} 
and equation  \eqref{invmellinZ} open up the possibility
to generalize some of that vast literature on  massless objects ($\mu=0$, flat space) to mass-deformed objects ($\mu \neq 0$, as in the plane gravitational wave) simply by representing them as inverse Mellin transforms
of the well-studied massless objects.

\subsection{An alternative formal representation}
\label{alternative}
The goal of this subsection is to introduce an additional parameter $w=A\tau+B$ in $\mathcal E_{1,\mu}$,
making it quasiperiodic in $z$ in the sense of  equation \eqref{E2}. To be more precise a function $f (w,z; \tau): \mathbb{C}^2\times \mathbb{H}\to\mathbb{C}$ is called \textit{quasiperiodic in $z$} if
\begin{equation}
f(w,z+1;\tau)= e^{2\pi i A} f(w,z ;\tau) \; , \quad
 f(w, z+\tau;\tau)=e^{-2\pi i B} f(w,z;\tau).
\end{equation}
For this, we take \eqref{invmellinZ} as starting point and use  \eqref{intrep} to introduce a nonzero first argument $\mathbb E_{s+1}(w,z;\tau)$. This turns out to provide a representation of the massive Kronecker--Eisenstein series as a power series in $\mu$. For this define
\[
\mathcal E_{1,\mu}(w,z;\tau) := \int_{c-i\infty}^{c+i\infty} (\pi \mu)^{-s} \Gamma(s)\mathbb E_{s+1}(w,z;\tau) ds.
\]
\begin{proposition}
\label{Etildeprop}
We have
\begin{equation}  \label{Etilderep}
\mathcal E_{1,\mu}(w,z;\tau) 
=e^{{2\pi i \over \tau_2} {\rm Im}(w \overline{z})} \sum_{n\geq0} {(-\pi \mu)^n \over n!} \mathbb E_{n}(z,w;\tau)
\end{equation}
with ${ \mathbb E}_{s+1}(w,z;\tau)$ defined in \eqref{intrep}. In particular $\mathcal E_{1,\mu}(w,z;\tau)$ is quasiperiodic in $z$.
\end{proposition}
\begin{proof}
For $z \notin {\mathbb Z} \tau + {\mathbb Z} $, the integral representation in equation \eqref{intrep}
extends $\mathbb E_s(w,z;\tau)$ to all $s$. 
We integrate along the vertical path of integration
in the inverse Mellin transform \eqref{invmellinZ} and close the contour along a semicircle at infinity around $s\rightarrow -\infty$.
Using the functional equation \eqref{E4}, we receive contributions from residues at the poles\footnote{This uses that $\mathbb E_s(w,z;\tau)$ for large positive $s$ does not ruin the exponential decay
due to the Gamma function for large negative $s$.
For example, for $w=0$, the author of \cite{Siegel} showed on p. 23, that  $\mathbb E_n(z,0;\tau)$ is majorized by $\zeta(2s)$ (with $\zeta$ the Riemann zeta-function),
and $\zeta(2s)$ decays as $2^{-s}$ as $s\rightarrow \infty$, whereas it would need to grow 
to ruin the decay due to the Gamma function. 
Similarly, for $w \neq 0$ on p.~42 of  \cite{Siegel} (set $g =0$ there),  $\mathbb E_n(z,w;\tau)$ was  majorized by the Hurwitz zeta function.}
of  $\Gamma(s)$ at $s=-n$ for $n\geq 0$, producing
 $e^{{2\pi i \over \tau_2} {\rm Im}(w \overline{z})} \frac{(-1)^n}{n!} \mu^n \mathbb E_{n}(z,w;\tau)$ as residue. 
\end{proof}
Note that if we try to use the double sum representation from equation \eqref{EsEarly} to perform the sum on $r$ and $\ell$ before the inverse Mellin transform \eqref{invmellinZ}, the double sum representation
does not converge for Re$(s)<0$, where we pick up the residues of $\Gamma(s)$. That is not a problem as long as we use the analytic continuation \eqref{intrep}.

Note that the power series in $\mu$ in \eqref{Etilderep} has no mixed $\mu^n \log (\mu)$ terms,
unlike Proposition \ref{E1muFourier} when the Bessel function is expanded in $\mu$. 

We can now use this expansion in $\mu$
to  connect to existing results in the literature, in particular \cite{DHoker:2015gmr}. 
For this purpose, we assume that $\mu<\frac12$,
and focus on  the special case $w=0$. We obtain from \eqref{Etilderep}, setting $\mathcal E_{1,\mu}(z;\tau) :=\mathcal E_{1,\mu}(z,0;\tau)$
\begin{align}   \nonumber
\mathcal E_{1,\mu}(z;\tau)    
&= {\mathbb E}_{0}(z,0;\tau)-\pi \mu \mathbb E_1(z,0;\tau)+\!\!\sum_{n\geq2}\!\! {(-\pi \mu)^n \over n}  \!\left(\tau_2 \over \pi\right)^{\!n} \!\!\sideset{}{^*}\sum_{r,\ell}\!\!{1 \over |z+r \tau+\ell|^{2n}} \\
&= {\mathbb E}_{1}(0,z;\tau)-\pi \mu \mathbb E_1(z,0;\tau)+\sideset{}{^*}\sum_{r,\ell}  \sum_{n\geq2} {1 \over n} \left({-\mu  \tau_2 \over  |z+r \tau+\ell|^2}\right)^n, \label{E1summu}
\end{align}
where the summation $\Sigma^*$ indicates that the sum runs over all $(r,\ell) \in\Z^2$ such that $w+r \tau+\ell\neq 0$.
This expression can be further manipulated to a double sum of logarithms, but convergence becomes somewhat more complicated and we therefore do not do so here. This last expression has the nice feature that the massless Green's function $\mathbb E_1 (0,z;\tau)$
appears as a separate first term, and the remaining part manifestly vanishes for $\mu=0$,
making the limit $\mu\rightarrow 0^+$ more manifest than in \eqref{Zlimit2}, Proposition \ref{E1muFourier}, and Proposition \ref{propmellin}.

In anticipation of the comparison to string theory literature in Section \ref{compare} below, note that 
\begin{eqnarray*}  \label{WDHoker}
\hspace{-0.9em} -\left[\partial^2_{ \mu}\left(\mathcal E_{1,\mu}(0,z;\tau)\right)\right]_{z=0}&=&\left[\sideset{}{^*}\sum_{r,\ell}{ \tau_2^2 \over (|z+r \tau+\ell|^2+\mu \tau_2)^2} \right]_{z=0}\\
&=&\sideset{}{^*}\sum_{r,\ell}{ \tau_2^2 \over (|r \tau+\ell|^2+\mu \tau_2)^2} \, . 
\end{eqnarray*} 
This is essentially a simpler version of the function $\mathcal W$ occurring in Section 4 of \cite{DHoker:2015gmr}, a generating function of (massless) modular graph functions.

\section{Comparison to string theory literature}
\label{compare}
This section is mainly intended for physics readers, or mathematicians who are curious why
physicists might be interested in the objects we study here. 

In string theory, just as in quantum field theory, it is natural to consider massive worldsheet fields,
either in a gravitational wave background, as  in \cite{Berenstein:2002jq}, or as a technical trick in flat space, as in \cite{Lee:2006pa}, including as generating function of modular graph functions \cite{DHoker:2015gmr}.

\subsection{Massive modular graph functions}
\label{modulargraph}

There is recent interest in modular graph functions \cite{DHoker:2015wxz},
that are constructed by integrating various combinations of (massless) Green's functions
over a fundamental domain for the action of the lattice $\Z \tau + \Z$. 
The simplest nontrivial example of a modular graph function is 
just the non-holomorphic Eisenstein series $\mathbb E_2(0;\tau)$, arising
from the integral over a fundamental domain $P$  of the product of two  Kronecker--Eisenstein series:
 $\mathbb E_1(z;\tau) \mathbb E_1 (-z;\tau)$,
viewed as Green's functions for the Laplace equation on the torus. 

A simple corollary of the considerations in this paper 
is that one can replace $\mathbb E_1(z;\tau)$ $\mathbb E_1(-z;\tau)$ with
$\mathcal E_{1,\mu} (z;\tau) \mathcal E_{1,\mu}(-z;\tau) $
to create a mass-deformed modular graph function that we might call \newline $\mathcal E_{1,1,\mu}(0;\tau)$. We integrate over the region $P$ with corners at the complex numbers $z=0$, $z=1$, $z=\tau$, and $z=\tau+1$, 
a fundamental domain for $\C / (\Z \tau + \Z)$. 
The integration is 
\begin{align}
\mathcal E_{1,1,\mu}(0;\tau)&= {1 \over \tau_2}\int_{ P} \mathcal E_{1,\mu}(z;\tau)\mathcal E_{1,\mu}(-z;\tau) d^2 z \nonumber\\
&\hspace{-1cm}=\int_0^{1} \int_0^{1}  
 \sideset{}{^*}\sum_{r_1,\ell_1}{  {\sqrt{\mu  \tau_2}} K_1(2\pi {\sqrt{\mu \over \tau_2}}  | r_1 \tau+\ell_1|) \over  | r_1 \tau+\ell_1|} e^{2\pi i(r_1\beta-\ell_1 \alpha)}\nonumber\\
&\hspace{2cm}\times  
 \sideset{}{^*}\sum_{r_2,\ell_2}{ {\sqrt{\mu \tau_2}} K_1(2\pi {\sqrt{\mu \over \tau_2}}  | r_2 \tau+\ell_2|) \over | r_2 \tau+\ell_2|} e^{2\pi i(\ell_2\alpha-r_2\beta)}d\alpha  d\beta\label{modgraph}
  \\
&\hspace{-1cm}= \mu\tau_2 \sideset{}{^*}\sum_{r_1,\ell_1} \sideset{}{^*}\sum_{r_2,\ell_2}{  K_1\left(2\pi {\sqrt{\mu \over \tau_2}}  | r_1 \tau+ \ell_1|\right) \over  | r_1 \tau+ \ell_1|}
 { K_1\left(2\pi {\sqrt{\mu \over \tau_2}}  | r_2 \tau+ \ell_2|\right) \over | r_2 \tau+ \ell_2|}\delta_{r_1,r_2}\delta_{\ell_1,\ell_2}
 \nonumber\\
 &\hspace{-1cm}=  \mu\tau_2 \sideset{}{^*}\sum_{r,\ell}  {K_1\left(2\pi {\sqrt{\mu \over \tau_2}}   | r \tau+ \ell|\right)^2 \over  | r \tau+ \ell|^2} \; . 
  \nonumber
\end{align}
For $\mu\rightarrow 0^+$, this clearly reduces to the massless modular graph function $\mathbb E_2(0;\tau)$.
In the third equality above, the integration over $\alpha,\beta$ produces factors like
 $\int_0^{1} e^{2\pi i(r_1-r_2)\alpha} \, d\alpha = \delta_{r_1,r_2}$ for $r_j\in\bbZ$, where $\delta_{r_1,r_2}=0$ unless $r_1=r_2$ in which case it equals 1.  This collapses
the two double sums in $\mathcal E_{1,\mu}(z;\tau)\mathcal E_{1,\mu}(-z;\tau) $  to a single double sum.
 
Note  that in the differential equations in the sections above, we view
$\tau$ and $z$ as independent variables. Here, $\tau$ is considered to be fixed and we integrate
over $z$, and are free to
change variables of integration from $z$ to $(\alpha,\beta)$
as  independent real variables, with Jacobian $\tau_2$. 

In this calculation, unlike in its massless counterpart, each double sum converges exponentially.
The undeformed $\mu=0$ eigenvalue equation is
$\Delta_{\tau,0}({\mathbb E}_{2}(0;\tau))=-2{\mathbb E}_{2}(0;\tau)$. 
For the $\mu$-deformed modular graph function, we have the following: 
\begin{proposition} \label{massiveE2}
We have 
\begin{equation*}\left(\Delta_{\tau,0}-2\mu\partial_{\mu}+\mu^2\partial_{\mu}^2\right){\mathcal E}_{1,1,\mu}(0;\tau) = -2{\mathcal E}_{1,1,\mu}(0;\tau).
\end{equation*}
\end{proposition}

\begin{proof}
The double sum in ${\mathcal E}_{1,1,\mu}(0;\tau)$ in \eqref{modgraph} converges absolutely
 so we can differentiate term by term. Each term has the form
is ${\mathcal E}_{1,1,\mu,r,\ell} (\tau)
=\frac{\mu\tau_2}{\omega} f(X_{\mu}(\tau))^2$ for some function $f$,
where $X_\mu(\tau):=2\pi \sqrt{\mu \omega(\tau) \over \tau_2}$ with $\omega (\tau) := |r\tau+\ell|^2$.
We find that 
\begin{multline*}
\frac{\Delta_{\tau,0}(\mathcal E_{1,1, \mu, r , \ell})(\tau)}{\mathcal E_{1,1, \mu, r , \ell}(\tau)}
= \frac{\pi \mu w(\tau)}{2 f (X_\mu (\tau))^2 \tau_2} \Big( 2 \sqrt{\tfrac{\tau_2}{\mu w(\tau)}} f(X_\mu(\tau)) f'(X_\mu(\tau)) \Big. \\
\Big. - 4 \pi \left(  f(X_\mu(\tau))  f''(X_\mu(\tau)) + f'(X_\mu(\tau))^2\right)
\Big).
\end{multline*}
The action of $\mu^2\partial_{\mu}^2$ is similar, namely
\begin{multline*}
\frac{\mu^2 \partial_\mu^2(\mathcal E_{1,1,r,\mu,\ell}(\tau))}{\mathcal E_{1,1,r,\mu,\ell}(\tau)} = \frac{\pi \mu w(\tau)}{2 \tau_2 f(X_\mu(\tau))^2} \left( \frac{6 \sqrt{\tau_2} f(X_\mu(\tau))f'(X_\mu(\tau))}{\sqrt{\mu w(\tau)}}\right.\\ \left.
+ 4\pi \( f(X_\mu(\tau))f''(X_\mu(\tau)) +  f'\(X_\mu(\tau)\)^2 \)\right).
\end{multline*}
Finally, we compute
\begin{equation*}
{-2\mu\partial_{\mu}({\mathcal E}_{1,1,\mu,r,\ell})(\tau) \over {\mathcal E}_{1,1,\mu,r,\ell}(\tau) }
=-2-4\pi \sqrt{\tfrac{\mu w}{\tau_2}} \frac{f'(X_\mu(\tau))}{f(X_{\mu}(\tau))}.
\end{equation*}
Combining gives the claim, without using any properties of $f$. 
\end{proof}
We are writing Proposition \ref{massiveE2} in the form above to make the relation to 
the undeformed modular graph function explicit. 
In terms of the massive Maass forms in Theorem \ref{Maassex},
we move the $\mu$-terms in the differential operator to the right, namely 
\begin{equation*} \Delta_{\tau,0}\left({\mathcal E}_{1,1,\mu}\right)(0;\tau) = \left(-\mu^2\partial_{\mu}^2+2\mu\partial_{\mu}-2\right){\mathcal E}_{1,1,\mu}(0;\tau) \; ,
\end{equation*}
so we identify $g_2(\mu)=-\mu^2$, $g_1(\mu)=2\mu$, and $g_0(\mu)=-2$. 
In the proof of Theorem \ref{Maassex}, we have
 $g_2(\mu)=-b^2\mu^2$, $g_1(\mu)=-\left(b^2+2bc+b\right)\mu$, and $g_0(\mu)=-c^2-c$, 
 so this corresponds to $(b,c)=(1,-2)$, since we are requiring $b>0$. 
But Theorem \ref{Maassex} holds for $f_{\mu,r,\ell}(\tau)= \frac{|r\tau+\ell|^{2c}}{\tau_2^c}h(\frac{\mu|r\tau+\ell|^{2b}}{\tau_2^b})$
for any $h$, and the proof of Proposition \ref{massiveE2} 
holds for any $f$, so there is the question of equivalence relations
as discussed in Section \ref{subsec:exp}. Indeed, the example we are discussing below Theorem \ref{Maassex}
has $K_2$ in the double sum for ${\mathcal E}_{2,\mu}(0;\tau)$, not $K_1^2$ as we have here in ${\mathcal E}_{1,1,\mu}(0;\tau)$
(hence the notation ``1,1''). 
There are many relations between Bessel functions, e.g.\  recurrence relations, that 
need to be taken into account in a systematic study of connections between massive Maass forms
and massive modular graph functions. The objective here is
just to provide one entry
point into those connections,
and we leave such investigations to the future. 

\subsection{Helmholtz equation}
\label{helmholtz}

The ordinary meaning of the ``massive'' Laplace equation
is the Helmholtz equation, i.e., its Green's function is a solution of 
\begin{equation*}
\left(2 \partial_{z}\partial_{\overline{z}} - m^2\right)G(z;\tau) = -\delta^{[2]}(z;\tau) \; . 
\end{equation*}
The differential operator $2 \partial_{z}\partial_{\overline{z}}$  is not invariant under
the Jacobi group, since e.g.\ under an $S$ transformation, $z\mapsto \frac z\tau$. Multiplying by $\tau_2$ brings the differential operator to our $\Delta_{z,0,0}=2\tau_2 \partial_{z}\partial_{\overline{z}}$,
and $m^2$ is completed to the invariant $\tau_2 m^2=\mu$.
Finding the solution for $G(z;\tau)$ as a double sum is a standard exercise
if $\tau_1=0$, see e.g. Appendix E of \cite{Lee:2006pa}. 
Using the basis functions $e^{2\pi i (r \beta -  \ell \alpha)}$ as
for ${\mathbb E}(z;\tau)$, we find that in coordinates
$z=\alpha \tau+\beta$,   we have $\Delta_{z,0,0}=2\tau_2 \partial_{z}\partial_{\overline{z}} =\frac1{\tau_2}(\tau_2^2\partial_{\beta}^2+\partial_{\alpha}^2)$, and
we multiply the differential operator with $\tau_2$ to find
\begin{equation}  \label{Gsum}
G(z;\tau): =  \sum_{(r, \ell)\in\Z^2}{e^{2\pi i (r \beta -  \ell \alpha)} \over 4\pi^2(r^2\tau_2^2+\ell^2) + \mu}.
\end{equation}
Integrating in $\alpha$ and $\beta$ gives
\begin{equation*}
\int_0^1 \int_0^1 G(z;\tau)\, d\alpha  d\beta = \frac{1}{\mu}.
\end{equation*} 
In physics, one may want to normalize this to unity
by multiplying $G(z;\tau) $ by $\mu$, but we do not do so here.
Note that if we remove the term $(r,\ell)=(0,0)$ and set $\mu=0$, then we obtain the Kronecker--Eisenstein series $\mathbb E_1(z;i \tau_2)$,
up to normalization.
Following e.g.\ Appendix E of \cite{Lee:2006pa}, we can write the summand in \eqref{Gsum} as an integral over a parameter $s$, then do modular inversion of  the sum over $\ell$, and evaluate the integral in $s$. For  $
\mu=0$
this gives a logarithm of $|\vartheta_1(z;i\tau_2)|$, as in 
 Kronecker's second limit formula \eqref{Estheta}. The result
for $\mu\neq 0$ is 
\begin{equation*}
G(z;\tau) =  \sum_{(r, \ell)\in\Z^2} {e^{-2\pi\tau_2
\sqrt{4\pi^2 r^2\tau_2^2+\mu} |\ell -\alpha|}e^{2\pi ir\beta} \over 2\sqrt{4\pi^2 r^2\tau_2^2 + \mu}} \; .
\end{equation*}
The inversion followed by the integration in  $s$ caused $r$ and $\ell$ to play different roles in the summation. This is not surprising,
as we single out the sum on $\ell$ by hand. 
In particular, if  we now let  $\mu \rightarrow0^+$, then in  terms with $r=0$, the  dependence on $\ell $ and $\beta$ is trivial which causes
the sum on $\ell$ to diverge,
even though that is not the case before those manipulations. 
One way to regularize  is to
 subtract the massive point-particle Green's function $G_{\rm particle}(\beta):=\frac{\cosh(m(|\beta|-\pi\tau_2))}{m \sinh(\pi m \tau_2)}$, which   is also (quadratically) divergent as $m\rightarrow 0$. Related discussions appear 
e.g.\ in \cite{Polchinski:1998rq} Appendix A. The appearance of the divergence is similar to the elementary sum
\begin{equation}  \label{point}
\sum_{\ell\geq1} {1 \over {\ell^2+m^2}} = {\pi \coth (\pi m) \over 2m}-{1 \over 2m^2}.
\end{equation}
The limit of the ``coth'' term does not exist  for $m\rightarrow 0$, but clearly the left-hand side is not divergent, and indeed the term $1 \over 2m^2$ subtracts the principal part. 
This is close in spirit to \eqref{E1summu}. There, we ``subtract'' the entire string Green's function $\mathbb E_1 (0,z;\tau)$, not just the point-particle Green's function $G_{\rm particle}$.  (``Subtract'' in quotation marks since just like in \eqref{point} we merely split up an expression in two parts that by themselves would be divergent, we did not subtract anything by hand.)


Going further back in the literature,
Sugawara \cite{Sugawara:2002rs}  computed
the partition function of the gravitational plane wave from the functional integral, in his equation (2.34) 
\begin{equation*}
Z_m(\tau)={1 \over \prod_{r, \ell  \in {\mathbb Z}} \tau_2(\tau_2^{-2}|r\tau+ \ell|^2+m^2)}
\end{equation*}
which implies that 
\begin{equation*}
\mathrm{Log} (Z_m(\tau))=- \sum_{r,\ell\in {\mathbb Z}}\log\left( \tau_2\left(\tau_2^{-2}|r\tau + \ell|^2+m^2\right)\right) \; . 
\end{equation*}
If we regulate this expression by differentiating twice with respect to $m$,
it looks like the double sum representation in \eqref{WDHoker}. But as expected,
this formal partition function calculation does not tell us how to traverse the divergence.  In \eqref{Etilderep},
we use a twist regularization and analytic continuation. 

 Finally, we note that there is an extensive literature on related topics. 
For example, in statistical field theory, similar objects were used
to compute renormalization group dependence of the conformal field theory central charge \cite{SI,Itzykson:1989sy}.
The Helmholtz equation with periodic boundary conditions has many other applications, e.g.\ in waveguide physics (see e.g.\ \cite{Linton}).

\section{Outlook}\label{Outlook}
It would be interesting to continue the discussion of modular graph functions (and forms)
from Section \ref{modulargraph}.
This should be feasible using the tools given in this paper. 

We now give another source for constructing more examples in the future. Note that we can interpret the proof of \eqref{Translaw} in \cite{BGG} as giving the expression
\begin{multline}\label{mess}
\log\(\mathcal F_{m}(\tau_2)\)
\\=2\pi \tau_2c_{m}+\frac{2\pi}{\tau_2}c_{m\tau_2}-\frac{1}{4}\!
\int_0^\infty \!\!\!\!e^{-\frac{\pi \tau_2 m^2}{s}}\!\(\vartheta_3\!(\I \tau_2 s)-1\)\(\!\vartheta_3\!\(\tfrac{\I s}{\tau_2}\)-1\!\)\!ds,\end{multline}
where $\vartheta_3(\tau):=\sum_{n\in\bbZ}q^{\frac{n^2}{2}}$. Equation \eqref{mess}
 is invariant under the simultaneous transformations $\tau_2\mapsto \tfrac{1}{\tau_2}$,
$m\mapsto m\tau_2$.
An obvious generalisation of \eqref{mess} is to replace $\vartheta_3$ with 
functions $f,g$ satisfying, for all $\tau_2 \in \R^+$,
$$f\(\tfrac{\I}{\tau_2}\)=\tau_2^k F(\I \tau_2),\quad g\(\tfrac{\I}{\tau_2}\)=t^k G(\I \tau_2)$$
for some $k\in\tfrac{1}{2}\Z$ and any functions $F,G$.
If $f$ and $g$ are Jacobi theta functions with characteristic,
then we recover the  special case $\tau=i \tau_2$  of $Z_{\alpha,\beta,m}(\tau)$. It might be possible to construct new massive Maass forms by a clever
choice of $f,g$, but we do not investigate this question in this paper.

Another question is: can we embed the massive deformation in some more general and familiar framework?
We believe  that there is a sense in which the massive deformations for SL${}_2(\bbR)$ ``sit'' inside automorphic forms
for more general Lie groups. We give some remarks on this without any pretense of rigor. 
The idea is that the results in the main text of this paper may be reproduced and generalized
by what  is called   warped Kaluza-Klein reduction in physics, for example from ordinary (massless) automorphic forms
on SL${}_3(\bbR)$ or Sp${}_2(\bbR)$
to massive automorphic forms
on SL${}_2(\bbR)$. 
To illustrate this, let us  view the usual Fourier expansion 
of the non-holomorphic Eisenstein series for SL${}_2(\bbR)$ as a warped Kaluza-Klein reduction to massive automorphic form on SL${}_2(\bbR)/(\tau_1\sim \tau_1+1)$.
This means we make an Ansatz for each nonzero-mode part of its Fourier expansion
that $\mathbb E_s\big|_{\tau_1 \, \text{piece}}  \propto f_m(\tau_2)e^{2\pi i m \tau_1}$, so
that the Laplacian in $\tau$ yields
\begin{eqnarray*}
\Delta_{\tau,0}(\mathbb E_s)\big|_{\tau_1 \, \text{piece}}
& \supset& -\tau_2^2\left(\partial_{\tau_2}^2 + \partial_{\tau_1}^2\right) 
f_m(\tau_2)e^{2\pi i m \tau_1}\\
&=&  -\tau_2^2\left(\partial_{\tau_2}^2 -4\pi^2 m^2\right) f_m(\tau_2)e^{2\pi i m \tau_1}
\end{eqnarray*}
which can be viewed as a ``massive''  one-dimensional differential operator in $\tau_2$ only. 
It is not truly the Helmholtz operator, since the ``mass'' when multiplying out is $4\pi^2 m^2\tau_2^2$, i.e., it depends on the vertical position $\tau_2$
in the upper half-plane in these coordinates.  
Demanding that $\mathbb E_s$ is an eigenfunction of $\Delta_{\tau,0}$ with eigenvalue $-s(s-1)$ produces
a differential equation purely in $\tau_2$,
\begin{equation*}
 \tau_2^2\left(\partial_{\tau_2}^2 -4\pi^2 m^2\right) 
f_m(\tau_2) = s(s-1)f_m(\tau_2)
\end{equation*}
which is a Bessel differential equation like \eqref{BesselODE}.
Demanding that $f_m(\tau_2)$ falls off as $\tau_2\rightarrow \infty$ yields specifically that
$f_m(\tau_2)= c_m \sqrt{\tau_2}K_{s-\frac12}(2\pi m \tau_2)$, with some constant $c_m \in \C$. This
is the standard result for the nonzero-mode Fourier expansion
of the usual SL${}_2(\bbR)$ Eisenstein series. 
We call this reduction from SL${}_2(\bbR)$ to SL${}_2(\bbR)/(\tau_1\sim \tau_1+1)$ ``warped'' essentially because the ``mass''
 $4\pi^2 m^2\tau_2^2$ depends on  $\tau_2$, unlike in the Helmholtz
equation. One  can find a  coordinate
system on the upper half-plane where the mass is constant, but in this (trivial) example it is not necessary to do so
to find $f_m(\tau_2)$. 

Pursuing this further is beyond the scope of this paper, but a first glimpse can be seen in 
Kiritsis--Pioline \cite{Kiritsis:1997em}, Appendix A, with the Eisenstein series for SL${}_3(\bbR)$.
The non-zero-mode terms in their equation (A.4) resemble our Proposition \ref{E1muFourier}. 
It would also be interesting to study the connection to Niebur--Poincar\'e series
 \cite{Angelantonj:2015rxa}, where Siegel--Narain theta functions provide multi-parameter families
 of automorphic forms. The parameters in the Siegel--Narain theta function comprise supersymmetric  Calabi--Yau moduli space with zero flux. The gravitational wave background has nonzero flux that produces the worldsheet mass term,
 so in that sense mass deformation in this paper can be thought of as a more drastic change
 than moving around in Calabi--Yau moduli space.  \appendix

\section{Review of the Kronecker--Eisenstein series $\mathbb E_s(w,z;\tau)$}
\label{Ereview}

In Appendix A, we review material of Appendix E in \cite{Berg:2014ama}.
We begin with the question of finding a (non-holomorphic)
function $\mathbb E_s: {\mathbb C}^2 \times {\mathbb H} \rightarrow   {\mathbb C}$ that
depends on a parameter $s\in\C$ and is doubly periodic on the torus in the first variable $w$
\begin{equation} \label{E1}
\mathbb E_s(w+1,z;\tau)= \mathbb E_s(w,z;\tau)  \; , \quad
\mathbb E_s(w+\tau,z;\tau)=\mathbb E_s(w,z;\tau) .
\end{equation}
In the second variable $z$ it should be quasiperiodic, namely
\begin{equation} \label{E2}
\mathbb E_s(w,z+1;\tau)= e^{2\pi i A}\mathbb E_s(w,z ;\tau),  \ \ 
\mathbb E_s(w, z+\tau;\tau)=e^{-2\pi i B}\mathbb E_s(w,z;\tau),
\end{equation}
where $w=A\tau +B$ ($A,B\in\R$). In physics, expressions corresponding to Feynman graphs are composed 
of Green's function of the Laplace equation. For the torus, allowing characteristics as in \eqref{E2}, we have
\begin{equation} \label{E3}
\Delta_{z,0,0} (G(z;\tau))=-2\pi e^{{2\pi i \over \tau_2}{\rm Im} (w\overline{z})} \delta^{[2]}(z;\tau).
\end{equation}
 The factor in front of $\delta^{[2]}(z;\tau)$ ensures compatibility with quasiperiodicity. 

We now define the \textit{Kronecker--Eisenstein series}\footnote{This has an additional $\Gamma(s)$ as compared to \cite{Berg:2014ama}, which makes it ``completed'' in the sense of the reflection formula \eqref{E4} below.} ($s\in\C$ with ${\rm Re}(s)\geq1$)
\begin{eqnarray}
\mathbb E_s(w,z;\tau)&:=&\Gamma(s)
\left( {\tau_2 \over \pi}\right)^s \sideset{}{^*}\sum_{r, \ell} {e^{{2\pi i\over \tau_2}{\rm Im}((w+r\tau+\ell)\overline{z})} \over |w+r\tau+\ell|^{2s}} \nonumber \\
&=&\Gamma(s)\left( {\tau_2 \over \pi}\right)^s \sideset{}{^*}\sum_{r,\ell} {e^{2\pi i (r+A)\beta-2\pi i(\ell+B)\alpha} \over |(r+A)\tau+\ell+B|^{2s}}\label{Es},
\end{eqnarray}
where the summation $\Sigma^*$ indicates that the sum runs over all $(r,\ell) \in\Z^2$ such that $w+r \tau+\ell\neq 0$.
It is not hard to see that $\mathbb E_{s}(w,z;\tau)$ satisfies the transformations in \eqref{E1} and \eqref{E2}. For the special case $w=0$ and $s=1$, it is also easy to see that $\mathbb E_1(0, z; \tau)$
does not quite satisfy \eqref{E3} but rather \eqref{LaplaceG}, with an extra
term that is required for the right-hand side to yield zero when integrated over a fundamental domain in $z$. When
$w\neq 0$ there is quasiperiodicity as in \eqref{E2}, and the Laplace operator no longer integrates to 
zero as it does for a doubly periodic function. That is why the extra term on the right-hand side is not needed in \eqref{E3}.

Note that, by Kronecker's second limit formula, $\mathbb E_1(0,z;\tau)$ can alternatively be represented as 
(see e.g.\  Chapter 20 of \cite{Lang} or Section 5 of \cite{Siegel})
\begin{equation} \label{Estheta}
\mathbb E_1(0,z;\tau) = -\log\left(\left|{\vartheta_1(z;\tau) \over \eta(\tau)}\right|^2\right)
+ {2\pi z_2^2 \over \tau_2}.
\end{equation}

The double sum \eqref{Es} is only absolutely convergent for Re$(s)>1$, but an analytic continuation to all complex $s$
can be found as an integral representation. 
If either $z$ or $w$ are  lattice points  there are additional pole terms at $s=0$ or $s=1$ in this integral representation, that are written out in Appendix E of \cite{Berg:2014ama}, but in this calculation we for simplicity stay away from lattice points.
The integral (Mellin) representation is then found from
\begin{equation*}
\mathbb E_s(w,z;\tau) = \int_0^{\infty} x^{s-1}
\sideset{}{^*}\sum_{r,\ell}e^{-{\pi x  \over \tau_2} |w +r\tau+\ell|^2 +{2\pi i \over \tau_2}{\rm Im}((w +r\tau +\ell)\overline{z})} dx,
\end{equation*}
which is valid for Re$(s)>1$. 
The sum is exponentially decaying for $x\to \infty$ but for $s\leq 1$ it is potentially divergent towards $x \to 0$. 
We follow the approach of Riemann, namely to split the integral into one from 0 to 1 and one from 1 to $\infty$, then use the modular transformation of a theta function on the first piece and change variables $x \mapsto \tfrac1x$ to obtain
\begin{multline} \label{intrep}
\mathbb E_s(w,z;\tau) =
e^{{2\pi i \over \tau_2}{\rm Im}(w\overline{z})} \int_1^{\infty} x^{-s}
\sideset{}{^*}\sum_{r,\ell}e^{-{\pi x  \over \tau_2} |z +r\tau +\ell|^2 +{2\pi i \over \tau_2}{\rm Im}((z +r\tau+ \ell )\overline{w})}dx\\
+\int_1^{\infty}x^{s-1}
\sideset{}{^*}\sum_{r,\ell}e^{-{\pi x \over \tau_2} |w +r\tau +\ell|^2 +{2\pi i \over \tau_2}{\rm Im}((w+r\tau+\ell)\overline z)} dx.
\end{multline} 
Although we are originally assuming that Re$(s)>1$, this integral representation gives an analytic continuation (in $s$) of $\mathbb E_s(w,z;\tau)$. Moreover it directly implies a symmetry under $s\mapsto 1-s$, 
\begin{equation} \label{E4}
\mathbb E_s(w,z;\tau)=e^{{2\pi i \over \tau_2}{\rm Im}(w\overline{z})}\mathbb E_{1-s}(z,w;\tau)
\end{equation}
which is the functional relation (reflection formula) for $\mathbb E_s(w,z;\tau)$. 
Note that the variables $w$ and $z$ are switched,
which gives a motivation to allow  $w\neq 0$ in the first place. Note that (as is familiar from discussions of L-functions, but here $\mathbb E_s(w,z;\tau)$ depends on $\tau$) the two sides of \eqref{E4} never simultaneously  have convergent double sum representations,
and that the reflection formula does not give any information on the behavior of the double sums 
in the strip $0<{\rm Re}(s)<1$, which is instead
provided by the integral representation. 

If $z\notin {\mathbb Z}\tau + {\mathbb Z}$, $\mathbb E_s(w,z;\tau)$ satisfies the partial differential equation
where we view $z$ and $\tau$ as independent variables,
\begin{equation}\label{DeltaEs}
\Delta_{z,0,0}(\mathbb E_s(w,z;\tau)) =-2\pi(s-1) \mathbb E_{s-1}(w,z;\tau) \; . 
\end{equation}

The  ``twisted'' (quasiperiodic) Kronecker--Eisenstein series has a factor in front of the delta function
that allows for quasiperiodicity:
\begin{equation} \label{deltatwisted}
\Delta_{z,0,0}\left(\mathbb E_1(w,z;\tau)\right) = -2\pi e^{{2\pi i\over \tau_2}{\rm Im}(w\overline{z})} \delta^{[2]}(z;\tau) \;. 
\end{equation}
A formal power series representation of the lattice delta function is 
\begin{equation}
\label{delta}
\delta^{[2]}(z;\tau)=\sum_{r,\ell \in {\mathbb Z}}e^{2\pi i (r\beta-\ell \alpha)}.
\end{equation}
The factor in front of the lattice delta function in \eqref{deltatwisted}
may seem inconsequential, since the right-hand side is zero away from lattice points and at $z=0$ the factor is one,
but it can be nontrivial at lattice points away from the origin. We do not
attempt to write nonzero-index distributions in detail here. 

Note that $\mathbb E_s(z,w;\tau)$ has weight zero, but it can be used to generate objects with non-zero weight, 
e.g.\ those called $\mathbb E_{s,k}(z,w;\tau)$ in Appendix E of \cite{Berg:2014ama}. We do not discuss them here.

\end{document}